\documentclass[12pt]{amsart}
\usepackage{amscd,amssymb,amsmath,multicol}
\usepackage{graphicx}
\newtheorem{thm}[equation]{Theorem}
\numberwithin{equation}{section}
\newtheorem{cor}[equation]{Corollary}

\newtheorem{lem}[equation]{Lemma}

\newtheorem{conj}[equation]{Conjecture}

\newtheorem{prop}[equation]{Proposition}

\newtheorem{fig}[equation]{Figure}

\begin{document}
\raggedbottom \voffset=-.7truein \hoffset=0truein \vsize=8truein
\hsize=6truein \textheight=8truein \textwidth=6truein
\baselineskip=18truept
\def\vareps{\varepsilon}
\def\mapright#1{\ \smash{\mathop{\longrightarrow}\limits^{#1}}\ }
\def\mapleft#1{\smash{\mathop{\longleftarrow}\limits^{#1}}}
\def\mapup#1{\Big\uparrow\rlap{$\vcenter {\hbox {$#1$}}$}}
\def\mapdown#1{\Big\downarrow\rlap{$\vcenter {\hbox {$\ssize{#1}$}}$}}
\def\on{\operatorname}
\def\spa{\on{span}}
\def\a{\alpha}
\def\bz{{\Bbb Z}}
\def\gd{\on{gd}}
\def\imm{\on{imm}}
\def\sq{\on{Sq}}
\def\ssp{\on{stablespan}}
\def\eps{\epsilon}
\def\br{{\Bbb R}}
\def\bc{{\Bbb C}}
\def\bh{{\Bbb H}}
\def\tfrac{\textstyle\frac}
\def\w{\wedge}
\def\b{\beta}
\def\A{{\cal A}}
\def\P{{\cal P}}
\def\zt{{\Bbb Z}_2}
\def\bq{{\Bbb Q}}
\def\ker{\on{ker}}
\def\coker{\on{coker}}
\def\u{{\cal U}}
\def\e{{\cal E}}
\def\exp{\on{exp}}
\def\wbar{{\overline w}}
\def\xbar{{\overline x}}
\def\ybar{{\overline y}}
\def\zbar{{\overline z}}
\def\ebar{{\overline e}}
\def\nbar{{\overline n}}
\def\mbar{{\overline m}}
\def\ubar{{\overline u}}
\def\et{{\widetilde E}}
\def\pt{{\widetilde P}}
\def\rt{{\widetilde R}}
\def\ni{\noindent}
\def\coef{\on{coef}}
\def\den{\on{den}}
\def\gd{{\on{gd}}}
\def\N{{\Bbb N}}
\def\Z{{\Bbb Z}}
\def\Q{{\Bbb Q}}
\def\R{{\Bbb R}}
\def\C{{\Bbb C}}
\def\Ah{\widehat{A}}
\def\Bh{\widehat{B}}
\def\Ch{\widehat{C}}
\def\Bin{\on{Bin}}
\def\xmin{x_{\text{min}}}
\def\xmax{x_{\text{max}}}
\title[Maximizing a combinatorial expression]
{Maximizing a combinatorial expression arising from crowd estimation}

\author{Donald M. Davis}
\address{Department of Mathematics, Lehigh University\\Bethlehem, PA 18015, USA}
\email{dmd1@lehigh.edu}
\date{October 29, 2009}


\maketitle
\begin{abstract} We determine, within 1, the value of $N$ for which $\sum\limits_i\binom{s_1}i\binom{s_2}N\binom{s_1}{N-i}\binom Ni$
achieves its maximum value. Here $s_1$ and $s_2$ are fixed integers. This problem arises in studying the most likely value of $|A\cup B\cup C|$ if $A$ and $C$ are
disjoint sets of cardinality $s_1$, and $|B|=s_2$. Attempting to remove the 1 unit of indeterminacy
leads to interesting conjectures about a family of rational functions.
 \end{abstract}

\section{Introduction}\label{intro}
The question considered here arises from problems involving estimating sizes of crowds.
You count sizes of certain subsets and want to estimate the size of the union.

The case considered here involves three sets $A$, $B$, and $C$ with the property that $A\cap C=\emptyset$, but
$B$ may intersect  the other sets. Suppose also that $|A|=s_1$, $|B|=s_2$, and $|C|=s_3$. What is the most likely value for $|A\cup B\cup C|$? This question was
suggested to the author by Fred Cohen, along with the mathematical model which we now present.

The assumption being made
is that any choice of $i$ people in $A\cap B$ and $j$ other people in $B\cap C$ is equally likely, for all
$i$ and $j$. For example, it is equally likely that $|A\cap B|=1$ with that person a specified person of $A$ and a specified person
of $B$ or that $|A\cap B|=3$ with those people a specified subset of $A$ and a specified subset of $B$. This can be formulated
in the following way.

Suppose $\Ah$, $\Bh$, and $\Ch$ are disjoint sets with $|\Ah|=s_1$, $|\Bh|=s_2$, and $|\Ch|=s_3$. The sample space consists of all 4-tuples
$$(A_1,B_1,B_2,C_2)\subset(\Ah,\Bh,\Bh,\Ch)$$ such that $|A_1|=|B_1|$, $|B_2|=|C_2|$, and $B_1\cap B_2=\emptyset$. Thus $A_1$ and $B_1$ correspond
to $A\cap B$ in the earlier formulation, and we have
$$|A\cup B\cup C|=s_1+s_2+s_3-|B_1|-|B_2|.$$
We assume that each element of the
sample space is equally likely. Let $E_{i,j}$ denote the event that $|B_1|=i$ and $|B_2|=j$. Then
$$|E_{i,j}|=\tbinom {s_1}i\tbinom{s_2}i\tbinom{s_2-i}j\tbinom{s_3}j.$$
If $E_N$ is the event that $|B_1|+|B_2|=N$, i.e., that $|A\cup B\cup C|=s_1+s_2+s_3-N$, then
$$|E_N|=\sum_{i+j=N}|E_{i,j}|.$$
Hence the most likely value of $|A\cup B\cup C|$ is $s_1+s_2+s_3-N$, where $N$ maximizes
$$\sum_i\tbinom{s_1}i\tbinom{s_2}i\tbinom{s_2-i}{N-i}\tbinom{s_3}{N-i}=\sum_i\tbinom{s_1}i\tbinom{s_2}N\tbinom{s_3}{N-i}\tbinom Ni.$$

We focus attention primarily on the case in which $s_1=s_3$. In this case we obtain in Corollary \ref{maincor} a simple formula for the maximizing $N$ within 1,
and in \ref{rfnmt2} a much-less-tractable formula which removes the indeterminacy. In Section \ref{sec3}, we attempt to obtain a more useful approximation to \ref{rfnmt2}, and in doing so we notice fascinating patterns in
a family of rational functions, but can only conjecture that these patterns persist. See Table \ref{oddd} and Conjecture
\ref{oddconj}. In Section \ref{genlcase},
we consider the general case when $s_1$ and $s_3$ need not be equal. Our results there are somewhat similar, but not so complete.

Our main theorem is
\begin{thm} \label{mainthm} Let $f_{s_1,s_2}(N):=\binom{s_2}N\sum\limits_i\binom{s_1}i\binom{s_1}{N-i}\binom Ni$ for integer values of $N$.
For each $s_1$ and $s_2$, there is an integer, which we denote by $g(s_1,s_2)$, such that $f_{s_1,s_2}(N)$ is an increasing function of $N$ for $N\le g(s_1,s_2)$, and a decreasing function
of $N$ for $N\ge g(s_1,s_2)$. Moreover,
\begin{equation}\label{1}g(s_1,s_2)= \biggl[2s_1+s_2+\tfrac32-\sqrt{4s_1^2+4s_1+(s_2+\tfrac12)^2}\biggr]+\delta\end{equation}
with $\delta=0$ or $1$.
\end{thm}

\begin{cor} The maximum value of $\sum\limits_i\binom{s_1}i\binom{s_2}N\binom{s_1}{N-i}\binom Ni$ occurs when
$$N=\biggl[2s_1+s_2+\tfrac32-\sqrt{4s_1^2+4s_1+(s_2+\tfrac12)^2}\biggr]+\delta$$
with $\delta=0$ or $1$.
The most likely value of $|A\cup B\cup C|$ in the situation discussed above, with $s_3=s_1$, is $\biggl\lceil \sqrt{4s_1^2+4s_1+(s_2+\tfrac12)^2}
-\frac32\biggr\rceil-\delta$.\label{maincor}
\end{cor}

It is conceivable that $f_{s_1,s_2}$ might achieve equal maxima at both $N$ and $N+1$. In such a case,
we accept either as an allowable value of $g(s_1,s_2)$.

To  illustrate the efficacy of our formula, we consider the
typical case $s_1=15$, and tabulate in Table \ref{t1} the actual
values of $g(15,s_2)$ for all $s_2$, and in Table \ref{t2} the five
values of $s_2$ for which $\delta=1$ in (\ref{1}).
Note how in these five cases the expression whose integer part appears in (\ref{1})
falls slightly short of the required value.

The following proposition generalizes the beginning and end of Table \ref{t1}. Theorem
\ref{mainthm} is true with $\delta=0$ in these cases.  Note also that the case $d=0$ of part (b) of Proposition \ref{easier}
shows that if $B$ is much larger than $A$ and $C$,
then the most likely occurrence is
that both $A$ and $C$ are contained in $B$.
\begin{prop}\label{easier} {}\quad
\begin{itemize}\item[a.] If $s_2\le\frac12(\sqrt{8s_1+9}-1)$, then $g(s_1,s_2)=s_2$.
\item[b.]  For all $s_1, s_2$, we have $g(s_1,s_2)\le 2s_1$. For $0\le d\le4$,
\begin{equation} \label{topfew}g(s_1,s_2)\ge 2s_1-d\text{ iff }s_2\ge\tfrac2{d+1}(s_1^2+s_1)-\tfrac{d+2}2.\end{equation}
\end{itemize}
\end{prop}
\newpage

\begin{table}[h]
\begin{center}
\caption{Values of $g(15,s_2)$}
\label{t1}
\begin{tabular}{c|l}
$s_2$&$g(15,s_2)$\\
\hline
$[1,6]$&$s_2$\\
$[7,10]$&$s_2-1$\\
$[11,12]$&$s_2-2$\\
$[13,15]$&$s_2-3$\\
$[16,17]$&$s_2-4$\\
$18$&$13$\\
$[19,21]$&$14$\\
$[22,23]$&$15$\\
$[24,26]$&$16$\\
$[27,29]$&$17$\\
$[30,33]$&$18$\\
$[34,37]$&$19$\\
$[38,42]$&$20$\\
$[43,48]$&$21$\\
$[49,55]$&$22$\\
$[56,64]$&$23$\\
$[65,76]$&$24$\\
$[77,92]$&$25$\\
$[93,117]$&$26$\\
$[118,157]$&$27$\\
$[158,238]$&$28$\\
$[239,478]$&$29$\\
$\ge479$&$30$
\end{tabular}
\end{center}
\end{table}

\begin{table}[h]
\begin{center}
\caption{Cases in which equality does not hold in (\ref{1}) when $s_1=15$
and $\delta=0$}
\label{t2}
\begin{tabular}{r|cc}
$s_2$&$g(15,s_2)$&$2s_1+s_2+\tfrac32-\sqrt{4s_1^2+4s_1+(s_2+\tfrac12)^2}$\\
\hline
$6$&$6$&$5.84$\\
$10$&$9$&$8.78$\\
$15$&$12$&$11.85$\\
$17$&$13$&$12.91$\\
$19$&$14$&$13.89$
\end{tabular}
\end{center}
\end{table}

Next we introduce the polynomials involved in the proof.
We will usually replace $s_1$ by $x$, both because it will occur as a
variable in polynomials, and so that we can use the notation $x_i=x(x-1)\cdots(x-i+1)$.
For a nonnegative integer $d$, define a polynomial $P_d(x)$ of degree $2d$ by
$$P_d(x)=\sum_{i=0}^d\frac{(x_i)^2(x_{d-i})^2}{i!(d-i)!}.$$

We will prove the following key result in Section \ref{combsec}.
\begin{lem}\label{Plem}
When $P_{d+1}(x)$ is divided by $P_d(x)$, the quotient is $\frac 2{d+1}x^2-\frac{2d}{d+1}x+\frac d2$.
Let $R_d(x)$ denote the remainder. If $f$ is as in \ref{mainthm}, then, if $x$ and $d$ are integers,
\begin{equation}\label{feq}f_{x,s_2}(2x-d)\ge f_{x,s_2}(2x-d-1)\text{ iff }s_2\ge\tfrac2{d+1}(x^2+x)-\tfrac{d+2}2+\tfrac{R_d(x)}{P_d(x)}.
\end{equation}
\end{lem}

In Section \ref{combsec}, we will also prove the following result, the proof of which is
less straightforward than that of Lemma \ref{Plem}.
\begin{lem}\label{hardlem} For $d\ge 5$ and $x>d/2$, $-0.5<R_d(x)/P_d(x)<0$.\end{lem}

Now we can prove the main theorem.
\begin{proof}[Proof of Theorem \ref{mainthm}] By Lemma \ref{hardlem}, increasing $d$ by 1 changes $R_d(x)/P_d(x)$ by at most 1/2, and clearly it decreases
$\frac2{d+1}(x^2+x)-\tfrac{d+2}2$ by more than 1/2.
Thus, for $x> d/2$,
the RHS of (\ref{feq}) is a decreasing function of $d$, at least for integer values of $d$. This, with (\ref{feq}),
implies the unimodality part of the theorem, with maximum of $f_{s_1,s_2}(N)$ occurring for $N=2s_1-d$ for the smallest
integer $d$ such that the RHS of (\ref{feq}) is satisfied.

We will show that  Lemma \ref{hardlem} also
implies that \begin{equation}\label{two}[2s_1-d_1]\le g(s_1,s_2)\le[2s_1-d_1]+1,\end{equation} where $d_1$ satisfies
$$s_2=\tfrac2{d_1+1}(s_1^2+s_1)-\tfrac{d_1+2}2.$$
This value is $d_1=-s_2-\tfrac32+\sqrt{4s_1^2+4s_1 +(s_2+\tfrac12)^2}$, yielding (\ref{1}).

To prove (\ref{two}), write $d_1=d_2-t$, with $0\le t<1$, and $d_2$ an integer. Thus $[2s_1-d_1]=2s_1-d_2$.
The RHS of (\ref{feq}) is satisfied using $d_2$ since, using \ref{hardlem} at the last step,
$$s_2=\tfrac2{d_1+1}(s_1^2+s_1)-\tfrac{d_1+2}2\ge\tfrac2{d_2+1}(s_1^2+s_1)-\tfrac{d_2+2}2\ge\tfrac2{d_2+1}(s_1^2+s_1)-\tfrac{d_2+2}2+\tfrac
{R_{d_2}(s_1)}{P_{d_2}(s_1)}.$$
Therefore $g(s_1,s_2)\ge 2s_1-d_2$. We will now show that the RHS of (\ref{feq}) is not satisfied using $d=d_2-2$, which implies
$g(s_1,s_2)\le 2s_1-d_2+1$,
hence completing the proof.

To see this, let $h(d)=s_2-\bigl(\frac2{d+1}(s_1^2+s_1)-\tfrac{d+2}2\bigr)$. Then $h(d_1)=0$ and $d_1-(d_2-2)>1$,
hence $h(d_2-2)<-\frac12$. Therefore when $d=d_2-2$, using Lemma \ref{hardlem} again, $s_2-\bigl(\frac2{d+1}(s_1^2+s_1)-\tfrac{d+2}2+\frac{R_d(s_1)}{P_d(s_1)})<0$, as desired.\end{proof}

In terms of $R_d/P_d$, we give in \ref{rfnmt2} a precise result about whether $\delta=0$ or 1 in Theorem \ref{mainthm}. The usefulness of this
is limited by the complicated nature of $R_d/P_d$. In Section \ref{sec3}, we discuss a very strong conjecture regarding $R_d/P_d$,
which, if proved, would make Theorem \ref{rfnmt2} more useful. See Theorem \ref{rfnmt}.
The evidence for this conjecture leads to remarkable conjectural patterns among some rational functions. See Table \ref{oddd} and Conjecture
\ref{oddconj}.
\begin{thm}\label{rfnmt2}  Let $$d_0=\biggl[\sqrt{4s_1^2+4s_1+(s_2+\tfrac12)^2}-s_2-\tfrac32\biggr].$$ Then (\ref{1}) is true with
$\delta=1$  iff
\begin{equation}\label{cond2} 0<\tfrac2{d_0+1}(s_1^2+s_1)-\tfrac{d_0+2}2-s_2\le-\frac{R_{d_0}(s_1)}{P_{d_0}(s_1)}.\end{equation}
\end{thm}
\begin{proof}
 Let $d_1=\sqrt{4s_1^2+4s_1+(s_2+\tfrac12)^2}-s_2-\tfrac32$. Then $s_2=\frac2{d_1+1}(s_1^2+s_1)-\frac{d_1+2}2$.

If $d_1$ is an integer, then $d_0=d_1$ and the $(>0)$-condition in (\ref{cond2}) is not satisfied, and the RHS of (\ref{feq}) is not satisfied when $d=d_1-1$
by an argument similar to the proof of \ref{mainthm}. Hence $g(s_1,s_2)=2s_1-d_1$, verifying Theorem \ref{rfnmt2} in this case.

If $d_1$ is not an integer, then (\ref{1}) is true with $\delta=1$ iff the RHS of (\ref{feq}) is satisfied using $d_0$ (since $2s_1-d_0=[2s_1-d_1]+1$),
but the RHS of (\ref{feq}) is exactly the $\le$-part of (\ref{cond2}). Note that the $(>0)$-part of (\ref{cond2}) is certainly satisfied in this case,
since the middle expression in (\ref{cond2}) equals 0 using $d_1$, and is a strictly decreasing function of $d$.\end{proof}

\section{Combinatorial proofs}\label{combsec}
In this section we prove Lemma \ref{Plem}, Proposition \ref{easier}, and Lemma
\ref{hardlem}.

\begin{proof}[Proof of Lemma \ref{Plem}]
Cancelling common factors in the binomial coefficients involving $s_2$, we find that
the LHS of (\ref{feq}) is equivalent to
\begin{eqnarray*}&&(s_2-2x+d+1)\sum\frac{(x!)^2(2x-d)!}{(i!)^2((2x-d-i)!)^2(x-i)!(d+i-x)!}\\
&\ge&(2x-d)\sum\frac{(x!)^2(2x-d-1)!}{(i!)^2((2x-d-1-i)!)^2(x-i)!(d+1+i-x)!}.\end{eqnarray*}
Cancelling $(2x-d)!$ and letting $j=x-i$, we obtain the equivalent condition
\begin{eqnarray*}&&(s_2-2x+d+1)\sum\frac{(x!)^2}{(x-j)!^2(x+j-d)!^2j!(d-j)!}\\
&\ge&\sum\frac{(x!)^2}{(x-j)!^2(x+j-d-1)!^2j!(d+1-j)!}.\end{eqnarray*}
Multiplying both sides by $(x!)^2$, the inequality becomes
$$(s_2-2x+d+1)\sum\frac{(x_j)^2(x_{d-j})^2}{j!(d-j)!}\ge\sum\frac{(x_j)^2(x_{d+1-j})^2}{j!(d+1-j)!},$$
and this readily yields
\begin{equation}\label{step1}f_{x,s_2}(2x-d)\ge f_{x,s_2}(2x-d-1)\text{ iff }s_2\ge\frac{P_{d+1}(x)}{P_d(x)}+2x-d-1.\end{equation}

Our next claim is that the leading terms of $P_d(x)$ are given by
\begin{equation}\label{3}P_d(x)=\frac{2^d}{d!}x^{2d}-\frac{2^{d-1}}{(d-2)!}x^{2d-1}+\frac{2^{d-3}}{3(d-2)!}(3d^2-5d+4)x^{2d-2}+\text{lower}.\end{equation}

The proof of (\ref{3}) makes frequent use of
\begin{equation}\label{short}\sum\tfrac1{i!(d-i)!}=\frac{2^d}{d!},\end{equation}
which is true since each side is the coefficient of $x^d$ in $e^x\cdot e^x=e^{2x}$. Then (\ref{short}) is exactly the coefficient of $x^{2d}$ in $P_d(x)$.
The coefficient of $x^{2d-1}$ is
$$-\sum_{i=0}^d\frac{i(i-1)+(d-i)(d-i-1)}{i!(d-i)!}=-2\sum_{i=0}^d\tfrac1{(i-2)!(d-i)!}=-\frac{2^{d-1}}{(d-2)!}.$$
We have used symmetry in the first step.
Note that the $i(i-1)$ comes as $\sum\limits_{j=0}^{i-1}2j$.

The next coefficient of $P_d(x)$ is obtained similarly, but involves much more work.
The coefficient of $x^{2d-2}$ in $(x_i)^2(x_{d-i})^2$ is
$$\sum_{j=0}^{i-1}j^2+\sum_{j=0}^{d-i-1}j^2+4\sum_{j=0}^{i-1}j\sum_{j=0}^{d-i-1}j+\sum_{1\le j_1< j_2< i}4j_1j_2+\sum_{1\le j_1< j_2<d-i}4j_1j_2.$$
Noting that $$2\sum_{1\le j_1< j_2< i}j_1j_2=\biggl(\sum_{j=1}^{i-1}j\biggr)^2-\sum_{j=1}^{i-1}j^2,$$
we obtain
\begin{eqnarray*}&&-\sum_{j=0}^{i-1}j^2-\sum_{j=0}^{d-i-1}j^2+(i-1)i(d-i-1)(d-i)+2\bigl(\tfrac{(i-1)i}2\bigr)^2+2\bigl(\tfrac{(d-i-1)(d-i)}2\bigr)^2\\
&=&-\tfrac{(i-1)i(2i-1)}6-\tfrac{(d-i-1)(d-i)(2d-2i-1)}6+(i-1)i(d-i-1)(d-i)\\
&&+\tfrac{i^2(i-1)^2}2+\tfrac{(d-i)^2(d-i-1)^2}2.
\end{eqnarray*}
Thus, using symmetry, the coefficient of $x^{2d-2}$ in
$\sum\limits_{i=0}^d\dfrac{(x_i)^2(x_{d-i})^2}{i!(d-i)!}$ is
$$\sum\tfrac{i(i-1)}{(i-2)!(d-i)!}+\sum\tfrac1{(i-2)!(d-i-2)!}-\tfrac13\sum\tfrac{2i-1}{(i-2)!(d-i)!}.$$

Next note that
\begin{eqnarray*}\sum\tfrac{i(i-1)}{(i-2)!(d-i)!}&=&\sum\tfrac{(j+1)(j+2)}{j!(d-2-j)!}\\
&=&2\frac{2^{d-2}}{(d-2)!}+3\frac{2^{d-3}}{(d-3)!}+\sum\frac{(j-1)+1}{(j-1)!(d-2-j)!}\\
&=&\frac{2^{d-1}}{(d-2)!}+3\frac{2^{d-3}}{(d-3)!}+\frac{2^{d-4}}{(d-4)!}+\frac{2^{d-3}}{(d-3)!}\\
&=&\frac{2^{d-4}}{(d-2)!}(8+8(d-2)+(d-2)(d-3))\\
&=&\frac{2^{d-4}}{(d-2)!}(d^2+3d-2).\end{eqnarray*}
Also,
$$\sum\tfrac{2i-1}{(i-2)!(d-i)!}=2\sum\tfrac1{(i-3)!(d-i)!}+3\sum\tfrac1{(i-2)!(d-i)!}=\frac{2^{d-2}}{(d-2)!}(d-2+3).$$
Hence the desired coefficient equals
\begin{eqnarray*}&&\frac{2^{d-4}}{(d-2)!}(d^2+3d-2)+\frac{2^{d-4}}{(d-4)!}-\frac{2^{d-2}}{3(d-2)!}(d+1)\\
&=&\frac{2^{d-3}}{3(d-2)!}(3d^2-5d+4),\end{eqnarray*}
as asserted in (\ref{3}).

Next we claim that  \begin{equation}\label{4}\frac{P_{d+1}(x)}{P_d(x)}=\tfrac 2{d+1}x^2-\tfrac{2d}{d+1}x+\tfrac d2+\frac{R_d(x)}{P_d(x)},\end{equation}
with $\deg(R_d(x))<2d$, the claim of the first part of Lemma \ref{Plem}. This can be discovered by division of polynomials, using (\ref{3}), but it is simpler just to verify that
\begin{eqnarray*}&&\biggl(\frac{2^d}{d!}x^{2d}-\frac{2^{d-1}}{(d-2)!}x^{2d-1}+\frac{2^{d-3}}{3(d-2)!}
(3d^2-5d+4)x^{2d-2}\biggr)
\biggl(\tfrac 2{d+1}x^2-\tfrac{2d}{d+1}x+\tfrac d2\biggr)\\
&=&\frac{2^{d+1}}{(d+1)!}x^{2d+2}-\frac{2^d}{(d-1)!}x^{2d+1}+\frac{2^{d-2}}{3(d-1)!}(3(d+1)^2-5(d+1)+4)x^{2d}+\text{lower.}
\end{eqnarray*}

Combining (\ref{step1}) and (\ref{4}), we obtain (\ref{feq}).\end{proof}

\begin{proof}[Proof of Proposition \ref{easier}]
(a). One can easily show that the bracketed
 expression in (\ref{1}) is always less than $s_2+1$, and equals $s_2$ if and only if
$s_2$ satisfies the hypothesis of \ref{easier}(a). Noting that $g(s_1,s_2)$ is $\ge$ the bracketed expression of (\ref{1}) by Theorem
\ref{mainthm}, and is $\le s_2$ since $f_{s_1,s_2}(s_2+1)=0$, part (a) follows.

(b). The first statement is true since $f_{s_1,s_2}(2s_1+1)=0$.

After cancelling common factors in the numerator and denominator, one can
compute that
$$\frac{-R_d(x)}{P_d(x)}=\begin{cases}0&d=0,1\\
(2x-1)/(3(2x^2-2x+1))&d=2\\
3(x-1)/(4(x^2-x+1))&d=3\\
6(2x^3-9x^2+3x-9)/(5(2x^4-8x^3+14x^2-5x+3))&d=4
\end{cases}$$
Thus (\ref{topfew}) follows immediately from (\ref{feq}) if $d=0$ or $1$.

Note that (\ref{feq}) is only meaningful if $2x>d$. If $d=2$, then $x\ge2$. For such $x$, $-0.2\le R_2(x)/P_2(x)<0$.
Since, for integer $x$, $\frac23(x^2+x)-2$ is either an integer or an integer plus 1/3, an integer $s_2$ satisfies
$$s_2\ge\tfrac23(x^2+x)-2+\tfrac{R_2(x)}{P_2(x)}\text{ iff }s_2\ge\tfrac23(x^2+x)-2,$$ and so (\ref{topfew}) follows from
(\ref{feq}).

A similar argument works for $d=3$ and $4$. For $d=3$, we have that $\tfrac2{d+1}(x^2+x)-\frac{d+2}2$ is an integer plus 1/2,
and for $x\ge2$, $-0.25\le R_3(x)/P_3(x)<0$. If $d=4$ and $x>3$, then $\tfrac2{d+1}(x^2+x)-\frac{d+2}2$ is an integer or an integer
plus $t$ with $t\ge0.2$, while $-0.12\le R_4(x)/P_4(x)<0$. If $d=4$ and $x=3$, (\ref{feq}) says $s_2\ge 2.8-0.55$, while the hypothesis
says $s_2\ge 2.8$. These are, of course, equivalent.
\end{proof}

\begin{proof}[Proof of Lemma \ref{hardlem}] We begin by removing common factors in $P_d(x)$ and $P_{d+1}(x)$.
Since parity of $d$ plays a role, we let $d=2b+\eps$ with $\eps\in\{0,1\}$. When
considering $R_{2b+\eps}(x)$, we let, for $\delta\in\{0,1\}$,
$$\pt_{2b+\eps+\delta}(x):=\frac{P_{2b+\eps+\delta}(x)}{\prod_{i=0}^{b-1+\eps}(x-i)^2}.$$
Note that $\pt_{2b+\eps+1}(x)/\pt_{2b+\eps}(x)$ has the same quotient as $P_{2b+\eps+1}(x)/P_{2b+\eps}(x)$,
while its remainder $\rt_{2b+\eps}(x)$ satisfies $\rt_{2b+\eps}(x)=R_{2b+\eps}(x)/\prod_{i=0}^{b-1+\eps}(x-i)^2$, and hence $\rt_{2b+\eps}(x)/\pt_{2b+\eps}(x)=R_{2b+\eps}(x)/P_{2b+\eps}(x)$.

To prove the lemma, we will prove
\begin{enumerate}
\item $\pt_{2b+\eps}(x+b)>0$ for $x>0$,
\item $\rt_{2b+\eps}(x+b)<0$ for $x>0$, and
\item $\rt_{2b+\eps}(x+b)+\frac12\pt_{2b+\eps}(x+b)>0$ for $x>0$.
\end{enumerate}

We have, with $c_{i,b}=1$ unless $i=b$, while $c_{b,b}=\frac12$,
\begin{eqnarray*}&&\tfrac{(2b+\eps+\delta)!}2\pt_{2b+\eps+\delta}(x+b)\\
&=&\begin{cases}\sum\limits_{i=0}^b c_{i,b}\tbinom{2b}i\prod\limits_{j=1}^{b-i}(x-j+1)^2\prod\limits_{j=b-i+1}^b(x+j)^2&\eps+\delta=0\\
\sum\limits_{i=0}^b \tbinom{2b+1}i\prod\limits_{j=0}^{b-i}(x-j)^2\prod\limits_{j=b-i+1}^b(x+j)^2&\eps+\delta=1\\
\sum\limits_{i=0}^{b+1}c_{i,b+1} \tbinom{2b+2}i\prod\limits_{j=0}^{b-i}(x-j-1)^2\prod\limits_{j=b-i+1}^b(x+j)^2&\eps+\delta=2.\end{cases}
\end{eqnarray*}

Part (1) is true since $\pt_{2b+\eps}(x+b)$ is a sum of nonnegative terms including the term $\prod\limits_{j=1-\eps}^b(x+j)^2$, which is positive
for $x>0$.

Next we consider (2) with $\eps=1$. We compute $q_{2b+1}(x+b)=\frac1{b+1}(x^2-x)+\frac12$. We have
\begin{eqnarray*}\tfrac{(2b+2)!}2\rt_{2b+1}(x+b)&=&\tfrac{(2b+2)!}2(\pt_{2b+2}(x+b)-q_{2b+1}(x+b)\pt_{2b+1}(x+b))\\
&=&\sum_{i=0}^b\biggl(\prod_{j=1}^{b-i}(x-j)^2\prod_{j=b-i+1}^b(x+j)^2\biggr)F_i,
\end{eqnarray*}
where
\begin{eqnarray*}F_i&=&\tbinom{2b+2}i(x-b+i-1)^2-(2b+2)\tbinom{2b+1}i(\tfrac1{b+1}(x^2-x)+\tfrac12)+\tfrac12\tbinom{2b+2}{b+1}\delta_{i,b}x^2\\
&=&\biggl(\tbinom{2b+1}{i-1}-\tbinom{2b+1}i+\delta_{i,b}\tbinom{2b+1}b\biggr)x^2-2\biggl(\tbinom{2b+1}{i-1}+(b-i)\tbinom{2b+2}i\biggr)x\\
&&+\tbinom{2b+2}i(b+1-i)^2-(b+1)\tbinom{2b+1}i\\
&=&\begin{cases}\tbinom{2b+1}{i-1}(x-(b+1-i))^2-\tbinom{2b+1}i(x+b-i)^2+\tbinom{2b+1}i(2(b-i)^2+b-2i)&i<b\\
\tbinom{2b+1}{b-1}(x-1)^2-b\tbinom{2b+1}b&i=b.\end{cases}
\end{eqnarray*}
Here $\delta_{i,b}$ is the Kronecker delta.

The first term of (the last form of) $F_i$ and second term of $F_{i-1}$, when multiplied by the appropriate double products, cancel.
Thus we obtain
\begin{eqnarray}&&\tfrac{(2b+2)!}2\rt_{2b+1}(x+b)\label{terms}\\
&=&\sum_{i=0}^{b-1}\biggl(\prod_{j=1}^{b-i}(x-j)^2\prod_{j=b-i+1}^b(x+j)^2\biggr)
\tbinom{2b+1}i(2(b-i)^2+b-2i)\nonumber\\
&&-\biggl(\prod_{j=1}^b(x+j)^2\biggr)b\tbinom{2b+1}b.\nonumber
\end{eqnarray}
One can easily prove that
$$\sum_{i=0}^{b-1}\tbinom{2b+1}i(2(b-i)^2+b-2i)-b\tbinom{2b+1}b=0.$$
Also $2(b-i)^2+b-2i>0$ iff $i<b+\frac12-\frac12\sqrt{2b+1}$. Note also that the double products are increasing with $i$ for $x>0$.
Thus our expression for $\tfrac{(2b+2)!}2\rt_{2b+1}(x+b)$ is of the form $\sum\limits_{i=0}^b\a_i\b_i$ with $0\le\a_1\le\cdots<\a_b$,
$\sum\b_i=0$, and $\b_i>0$ iff $i<i_0$. Such a sum is negative.

The proof for (2) when $\eps=0$ is extremely similar. We have $q_{2b}(x+b)=\frac2{2b+1}x^2+\frac b{2b+1}$. Then
$${\tfrac{(2b+1)!}2}\rt_{2b}(x+b)=\sum_{i=0}^b\prod_{j=1}^{b-i}(x-j+1)^2\prod_{j=b-i+1}^b(x+j)^2\cdot F_i,$$
where now
\begin{eqnarray*}F_i&=&\tbinom{2b+1}i(x-b+i)^2-c_{i,b}\tbinom{2b}i(2x^2+b)\\
&=&\begin{cases}\tbinom{2b}{i-1}(x-(b-i))^2-\tbinom{2b}i(x+b-i)^2+\tbinom{2b}i(2(b-i)^2-b)&i<b\\
\tbinom{2b}{b-1}x^2-b\tbinom{2b-1}{b-1}&i=b.\end{cases}
\end{eqnarray*}
The rest of the argument follows exactly the same steps as in the last paragraph of the above proof of the case $\eps=1$,
using $\sum_{i=0}^{b-1}\binom{2b}i(2(b-i)^2-b)-b\binom{2b-1}{b-1}=0$.

The proof of (3) is essentially the same, except that we are subtracting 1/2 from $q_{2b+\eps}(x+b)$. The effect, when $\eps=1$, is to
add $(b+1)\binom{2b+1}i$ to $F_i$. The replacement for (\ref{terms}) is
\begin{eqnarray*}&&\tfrac{(2b+2)!}2(\rt_{2b+1}(x+b)+\tfrac12\pt_{2b+1}(x+b))\\
&=&\sum_{i=0}^{b-1}\biggl(\prod_{j=1}^{b-i}(x-j)^2\prod_{j=b-i+1}^b(x+j)^2\biggr)
\tbinom{2b+1}i(2(b-i)^2+2b-2i+1)\\
&&+\biggl(\prod_{j=1}^b(x+j)^2\biggr)\tbinom{2b+1}b
\end{eqnarray*}
which is clearly positive for $x>0$. The proof when $\eps=0$ is similar.
\end{proof}

\section{Conjectures about $R_d(x)/P_d(x)$}\label{sec3}
In Theorem \ref{rfnmt2}, we determined the precise value of our focal function $g(s_1,s_2)$ in terms of $R_d(s_1)/P_d(s_1)$.
In order to make this result useful, we need better information about the family of functions $R_d(x)/P_d(x)$.
In this section, we present several conjectures about this family of functions, one of which is supported by remarkable
patterns. See Table \ref{oddd} and Conjecture
\ref{oddconj}. We also discuss their implications.

We now state the simplest of these conjectures.
\begin{conj}\label{conj2} If $d\ge5$ and $x\ge\frac12(d+\sqrt{d+2})$, then $-\frac{R_d(x)}{P_d(x)}\ge\frac{.995(d-2)}{2x+d-2}$.
\end{conj}

The implication of this conjecture is given by the following theorem.
\begin{thm}\label{rfnmt} Assume Conjecture \ref{conj2} and $s_2>\frac12(\sqrt{8s_1+9}-1)$. Let
$$d_0=\biggl[\sqrt{4s_1^2+4s_1+(s_2+\tfrac12)^2}-s_2-\tfrac32\biggr]\ge5.$$ Then (\ref{1}) is true with
$\delta=1$  if
\begin{equation}\label{cond} 0<\tfrac2{d_0+1}(s_1^2+s_1)-\tfrac{d_0+2}2-s_2\le\frac{0.995(d_0-2)}{2s_1+d_0-2}.\end{equation}
\end{thm}
For $s_1\le38$, the only cases in which (\ref{1}) is true with $\delta=1$ which are missed by this theorem are $(s_1,s_2)=(6,4)$, $(18,56)$, $(36,16)$,
and $(38,155)$. The significance of the 0.995 in \ref{conj2} is that it is, to three decimal places, the largest number for which the inequality appears to be true.

\begin{proof}[Proof of Theorem \ref{rfnmt}]
Let $d=\sqrt{4s_1^2+4s_1+(s_2+\tfrac12)^2}-s_2-\tfrac32$. The theorem is vacuously true if $d_0=d$ because $\frac2{d+1}(s_1^2+s_1)-\frac{d+2}2-s_2=0$.
So we assume $d$ is not an integer. Then $2s_1-d_0=[2s_1-d]+1$, and so, using (\ref{feq}), the assertion  that (\ref{1}) is true with $\delta=1$
can be stated as
$$s_2\ge\tfrac2{d_0+1}(s_1^2+s_1)-\tfrac{d_0+2}2+\tfrac{R_{d_0}(s_1)}{P_{d_0}(s_1)}.$$
This will follow from (\ref{cond}) and our assumption of Conjecture \ref{conj2} once we know that $s_1\ge\frac12(d_0+\sqrt{d_0+2})$.

Since $\sqrt{4s_1^2+4s_1+(s_2+\tfrac12)^2}-s_2-\tfrac32$ is a decreasing function of $s_2$, and $d_0<d$, it suffices to prove $s_1\ge\frac12(d+\sqrt{d+2})$
if
\begin{eqnarray*}d&=&\sqrt{4s_1^2+4s_1+(\tfrac12\sqrt{8s_1+9})^2}-\tfrac12\sqrt{8s_1+9}-1\\
&=&2s_1+\tfrac12-\sqrt{2s_1+\tfrac94}.\end{eqnarray*}
Solving the latter equation for $s_1$ yields exactly $s_1=\frac12(d+\sqrt{d+2})$.
\end{proof}

Extensive {\tt Maple} calculation led the author to expect that, for $d\ge5$ and $x>d/2$,
\begin{equation}\label{5}\frac{R_d(x)}{P_d(x)}\approx \frac{-(d-2)}{2x+d-2}.\end{equation}

To understand how good is the approximation (\ref{5}), we consider the ratio of the two sides, using the reduced
versions $\rt$ and $\pt$.
As this will be close to 1, we study
\begin{eqnarray}Q_d(x)&:=&1-\frac{\rt_d(x)/\pt_d(x)}{-(d-2)/(2x+d-2)}\nonumber\\
&=&\frac{(d-2)\pt_d(x)+(\pt_{d+1}(x)-q_d(x)\pt_d(x))(2x+d-2)}{(d-2)\pt_d(x)},\label{Q_d}\end{eqnarray}
where $q_d(x)=\tfrac 2{d+1}x^2-\tfrac{2d}{d+1}x+\tfrac d2$ is the quotient in (\ref{4}).
We would like to prove that $Q_d(x)\approx0$ in some sense, when $x>d/2$.

Similarly to the methods in deriving (\ref{3}), we can show that
\begin{equation}\label{6}\lim_{x\to\pm\infty}Q_d(x)=\tfrac{-2}{(d+1)(d-2)}.\end{equation}
To see this, note that the numerator and denominator of (\ref{Q_d}) are both polynomials of degree $2[\frac d2]$.
The desired limit in (\ref{6}) is the ratio of their leading coefficients.
We omit the details in this computation.

We begin by considering $Q_{25}(x)$. It is a ratio of two polynomials of degree 24.  {\tt Maple} computes
that the derivative of $Q_{25}(x)$ is 0 only at $x\approx4.0409$ and $60.50336$. Moreover,
{\tt Maple} plots the graph of $Q_{25}(x)$, which turns out to look something like the
rough sketch in the top half of Figure \ref{Q245graph}. This sketch is not at all to scale.
We are particularly interested in the values for $x\ge13$. As $x$ increases from $13$ to $60.5$,
$Q_{25}(x)$ decreases from $0.01672$ to $-0.006867$.

\begin{fig}
\label{Q245graph}
\begin{center}
\includegraphics[width=6in,height=6in]{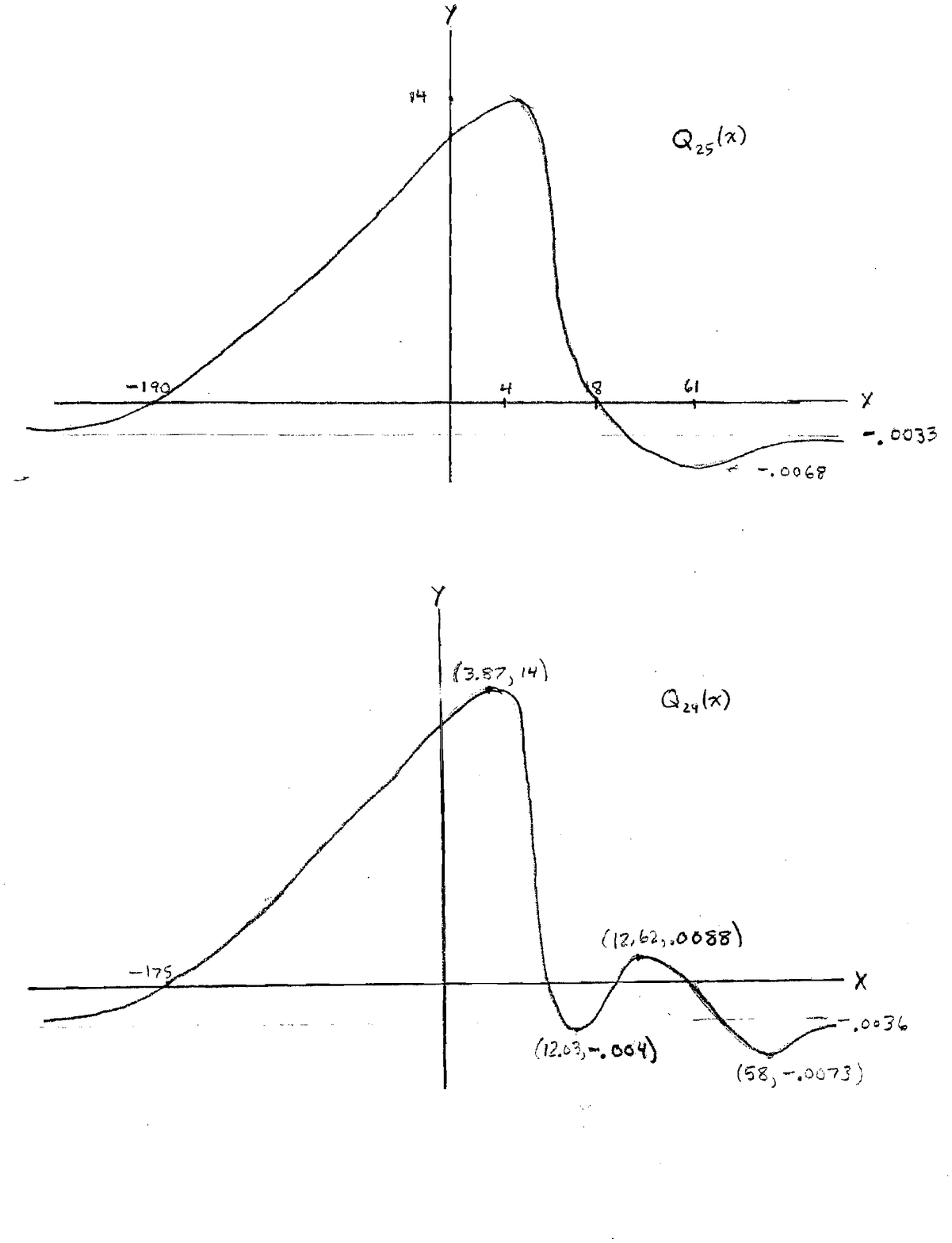}
\end{center}
\end{fig}

The amazing observation is that, for all odd $d$, $Q_d(x)$ apparently has a form very similar to that
in the top half of Figure \ref{Q245graph}, with only one local maximum and one local minimum, which are the
absolute maximum and minimum. Note that $Q_d(x)$ is a ratio of two
polynomials of degree $d-1$, yet it apparently has this simple form for all odd $d$.

In Table \ref{oddd}, we tabulate for odd $d$ satisfying $5\le d\le 61$, the values, $\xmax$ and $\xmin$, of $x$
where the derivative $Q_d'(x)$ equals $0$, the values of $Q_d(x)$ at these points, which will be absolute maximum and minimum values,
and the limiting value $Q_{\text{lim}}$ of $Q(x)$ as $x\to\pm\infty$.

\begin{scriptsize}
\begin{table}
\begin{center}
\caption{Max, min, and lim of $Q_{d}(x)$ when $d$ is odd}
\label{oddd}
\begin{tabular}{c|ccccc}
$d$&$\xmax$&$Q_d(\xmax)$&$\xmin$&$Q_d(\xmin)$&$Q_{\text{lim}}$\\
\hline

  $5$&$ 0.6874697648$&$ 3.069437405$&$ 10.48133802$&$ -0.1360199330$&$-0.1111111111$\\
  $7$&$ 1.044428415$&$ 4.216907295$&$ 15.48679761$&$ -0.06528852351$&$-0.05000000000$\\
  $9$&$ 1.371314912$&$ 5.374255987$&$ 20.49141428$&$ -0.03966514187$&$ -0.02857142857$\\
  $11$&$ 1.706180162$&$ 6.540442715$&$ 25.49465519$&$ -0.02724322575$&$-0.01851851852$\\
$13$&$ 2.040065994$&$ 7.712042960$&$ 30.49698936$&$ -0.02018382313$&$ -0.01298701299$\\
$15$&$ 2.373641870$&$ 8.886960397$&$ 35.49873537$&$ -0.01574295269$&$ -0.009615384615$\\
$17$&$ 2.707177063$&$ 10.06413875$&$ 40.50008573$&$ -0.01274397486$&$-0.007407407407$\\
$19$&$ 3.040658416$&$ 11.24291343$&$ 45.50115931$&$ -0.01060966757$&$  -0.005882352941$\\
$21$&$ 3.374105921$&$ 12.42285407$&$ 50.50203244$&$ -0.009028116216$&$ -0.004784688995$\\
$23$&$ 3.707530371$&$ 13.60364584$&$ 55.50275603$&$ -0.007817967242$&$  -0.003968253968$\\
$25$&$ 4.040938032$&$ 14.78513590$&$ 60.50336525$&$ -0.006867535404$&$   -0.003344481605$\\
$27$&$ 4.374333149$&$ 15.96714209$&$ 65.50388509$&$ -0.006104769071$&$  -0.002857142857$\\
$29$&$ 4.707718646$&$ 17.14958715$&$ 70.50433380$&$ -0.005481366771$&$ -0.002469135802$\\
 $ 31$&$ 5.041096602$&$ 18.33236558$&$ 75.50472499$&$-0.004963889096$&$ -0.002155172414$\\
  $33$&$ 5.374468538$&$19.51542724$&$ 80.50506904$&$ -0.004528540629$&$ -0.001897533207$\\
  $35$&$ 5.707835592$&$ 20.69872494$&$ 85.50537395$&$-0.004157982647$&$ -0.001683501683$\\
 $37$&$ 6.041198635$&$ 21.88222143$&$90.50564603$&$ -0.003839317503$&$ -0.001503759398$\\

  $39$&$ 6.374558339$&$ 23.06588687$&$ 95.50589030$&$ -0.003562775224$&$ -0.001351351351$\\
$41$&$ 6.707915237$&$ 24.24969671$&$ 100.5061108$&$  -0.003320834926$&$ -0.001221001221$\\
        $43$&$ 7.041269754$&$ 25.43363246$&$ 105.5063109$&$ -0.003107623257$&$ -0.001108647450$\\

$45$&$ 7.374622231$&$ 26.61767652$&$ 110.5064932$&$  -0.002918493910$&$ -0.00101112234$\\

    $47$&$ 7.707972951$&$27.80181575$&$  115.5066600$&$ -0.002749728194$&$ -0.000925925926$\\

  $49$&$ 8.041322143$&$ 28.98603869$&$ 120.5068132$&$-0.002598318198$&$ -0.00085106383$\\
  $51$&$ 8.374670000$&$ 30.17033570$&$  125.5069545$&$-0.002461807379$&$-0.00078492936$\\
  $53$&$ 8.708016684$&$ 31.35469854$&$  130.5070851$&$-0.002338171717$&$-0.00072621641$\\
  $55$&$ 9.041362332$&$ 32.53912015$&$  135.5072062$&$-0.002225729996$&$-0.00067385445$\\
  $57$&$ 9.374707059$&$ 33.72359445$&$  140.5073188$&$-0.002123075281$&$-0.00062695925$\\
  $59$&$ 9.708050964$&$ 34.90811616$&$  145.5074238$&$-0.002029022010$&$-0.00058479532$\\
  $61$&$10.041394134$&$ 36.09268071$&$  150.5075219$&$-0.001942564749$&$-0.00054674686$

\end{tabular}
\end{center}
\end{table}
\end{scriptsize}

When $d$ is even, the functions $Q_d$  fall into almost the same pattern, except that they have an additional wiggle
between $\frac d2$ and $\frac d2+1$. For example, a schematic graph of $Q_{24}(x)$ is given
in the bottom half of Figure \ref{Q245graph}. Note that the graph is drawn wildly out of scale. Similarly to
$Q_{25}$, it has an absolute maximum at $x=3.87$ and an absolute minimum at $x=58.003$.
But instead of decreasing steadily between these, it has an additional single local minimum and local
maximum which occur between $x=12$ and 13. {\tt Maple} calculations strongly suggest that for all
even $d$, the graph of $Q_d$ will have a form similar to that in the bottom half of Figure \ref{Q245graph},
and that the positions and values of the maxima and minima will have patterns extremely similar to
those for odd $d$ in Table \ref{oddd}.  We will not pursue those here, as we prefer to concentrate on
the simpler situation when $d$ is odd.

The reader will immediately be struck by the pattern in Table \ref{oddd}, which seems especially striking for $\xmin(d)$.
We have extended these calculations through $d=151$, using 80 digits of accuracy in {\tt Maple}. Then for $k=3,\ldots,10$,
we have found the real numbers $c_0,\ldots,c_k$ which satisfy
\begin{equation}\label{cs}\xmin(d)={\tfrac52}d-2+\sum_{i=0}^k\frac{c_i}{(d-1)^i},\qquad d=51,\, 61,\ldots,51+10k.\end{equation}
Each $c_i$ seems to stabilize as $k$ increases in (\ref{cs}). Moreover, using the formula (\ref{cs}) for $\xmin(d)$ derived using just a few values
of $d$
gives agreement with computed values of $\xmin(d)$ for all odd values of $d$ to an increasing number of decimal places as $k$ increases.
In addition, we have, with $k=10$, $c_0$ equals, to 19 decimal places, $.0104166666666666666\approx 1/96$. 

\begin{conj}\label{oddconj} For odd $d\ge5$, there are numbers $\xmax$ and $\xmin$ such that
$Q_d(x)$ is decreasing for $\xmax\le x\le\xmin$, and increasing elsewhere. There are real numbers $c_i$ for $i\ge1$ such that
$$\xmin(d)={\tfrac52}d-2+{\tfrac1{96}}+\sum_{i=1}^\infty\frac{c_i}{(d-1)^i}.$$
The initial digits of $c_1,\ldots,c_6$ are $-.176504629629629$, $.16562740498$, $.20004439$, $.291872$,  $.3215$, and $.28$.
\end{conj}
The 3-digit repetend in $c_1$ leads one to guess that $c_1=-305/12^3$ and that the $c_i$ are all rational numbers.

We have performed similar analyses for $\xmax$ and $Q_d(\xmin)$. The initial terms are apparently $\xmax=\frac d6-\frac18-\frac1{64(d-1)}$ and
$Q_d(\xmin)=-\frac1{12(d-1)}$. The series for $Q_d(\xmin)$ seems  to converge more slowly than the others.

We wish to emphasize that we cannot prove that $Q_d(x)$ for odd $d$ has a unique maximum and minimum. This is all based
on {\tt Maple} calculations obtained by setting its derivative equal to 0, where $Q_d$ is a ratio of two polynomials of degree
$d-1$.

Conjecture \ref{conj2} is equivalent to saying that $Q_d(x)\le .005$ for $x\ge\frac12(d+\sqrt{d+2})$. This latter statement
would follow from Conjecture \ref{oddconj} expanded to include a formula for $Q_d(\xmin)$ and to
include even values of $d$, together with a proof that $Q_d(\frac12(d+\sqrt{d+2}))\le.005$. Some justification for
this conjecture is given by Table \ref{tbl3}, which also shows why we use .005.

\begin{table}
\begin{center}
\caption{Evidence for Conjecture \ref{conj2}}
\label{tbl3}
\begin{tabular}{r|l}
$d$&$Q_d(\frac12(d+\sqrt{d+2}))$\\
\hline
                       $ 5$&$ -.08632297$\\

                        $6$&$ -.0563567189$\\

                        $7$&$ -.031250000$\\

                       $ 8$&$ -.021426047$\\

                        $9$&$ -.012806353$\\

                       $10$&$ -.008467636$\\

                       $11$&$ -.004769088$\\

                       $12$&$ -.002577860$\\

                      $13$&$ -.000764689$\\

                       $14$&$ 0.0004262575$\\

                       $15$&$ 0.001391319$\\

                       $16$&$ 0.002066212$\\

                       $17$&$ 0.002604694$\\

                       $18$&$ 0.002994356$\\

                       $19$&$ 0.003300396$\\

                       $20$&$ 0.00352431$\\

                       $21$&$ 0.00369610$\\

                       $22$&$ 0.00382002$\\

                       $23$&$ 0.00391095$\\

                       $24$&$ 0.00397283$\\

                       $25$&$ 0.00401357$\\

                       $26$&$ 0.00403620$\\

                       $27$&$ 0.004045174$\\

                       $28$&$ 0.004042664$\\

                       $29$&$ 0.0040313148$\\

                       $30$&$ 0.004012644$\\
                       $31$&$0.003988285$
\end{tabular}
\end{center}
\end{table}

There is one value of $Q_d(x)$, occurring just before the crucial range $x>d/2$, for which the value of
$Q_d(x)$ is easily determined. This is given in the following result, whose easy proof we omit.
\begin{prop} If $d$ is odd, then $Q_d(\frac{d-1}2)=1$. If $d$ is even, then $Q_d(\frac d2)=\frac{-2}{(d+1)(d-2)}$.
\end{prop}

\section{The general case ($s_1$ and $s_3$ not necessarily equal)}\label{genlcase}
In this section, we present our analysis of the general case, which is similar to, but not nearly so thoroughly
developed as, the case $s_1=s_3$ considered in the preceding sections.

For arbitrary $s_1$, $s_2$, and $s_3$, now let $$f_{s_1,s_2,s_3}(N):=\tbinom{s_2}N\sum_i\tbinom{s_1}i\tbinom{s_3}{N-i}\tbinom Ni.$$
{\tt Maple} suggests that for any values of the $s_i$ this $f_{s_1,s_2,s_3}$ is a unimodal function of $N$. If so, we can find the value of $N$ at which
$f$ achieves a maximum by an analysis extremely similar to that employed in the case $s_1=s_3$.

The formula for $f$ is symmetric in $s_1$ and $s_3$. We write $s_1=x$ and $s_3=x+\Delta$, $\Delta\ge0$.
Let $d=s_1+s_3-N$, and
$$P_{\Delta,d}(x)=\sum j!(d-j)!\tbinom{x+\Delta}j^2\tbinom x{d-j}^2=\sum\frac{((x+\Delta)_j)^2(x_{d-j})^2}{j!(d-j)!}.$$
Generalizing (\ref{4}), which is the case $\Delta=0$, we have
\begin{equation}\label{q2}\frac{P_{\Delta,d+1}}{P_{\Delta,d}}=\frac{2x^2-2(d-\Delta)x+\tfrac12d(d+1-2\Delta)+\Delta^2}{d+1}
+\frac{R_{\Delta,d}(x)}{P_{\Delta,d}(x)}.\end{equation}

The easy generalization of (\ref{feq}) is
\begin{eqnarray}f_{x,s_2,x+\Delta}(2x+\Delta-d)&\ge&f_{x,s_2,x+\Delta}(2x+\Delta-d-1)\nonumber\\
&\text{iff}&\nonumber\\
s_2&\ge&\tfrac{2x^2+2(\Delta+1)x+\Delta^2+\Delta}{d+1}-\tfrac{d+2}2+\tfrac{R_{\Delta,d}(x)}{P_{\Delta,d}(x)},\label{Delta}
\end{eqnarray}
where $R_{\Delta,d}(x)$ is the remainder in (\ref{q2}). If we assume this remainder is negligible, then imposing equality in
(\ref{Delta}) and recalling $s_1=x$ and $s_3=x+\Delta$ yields
$$d=\sqrt{2s_1^2+2s_1+2s_3^2+2s_3+(s_2+\tfrac12)^2}-s_2-\tfrac32$$
and \begin{equation}\label{Nform}N=s_1+s_2+s_3+\tfrac32-\sqrt{2s_1^2+2s_1+2s_3^2+2s_3+(s_2+\tfrac12)^2},\end{equation}
as nice a generalization of \ref{rfnmt2} and (\ref{1}) as one could possibly desire. This yields
$$\sqrt{2s_1^2+2s_1+2s_3^2+2s_3+(s_2+\tfrac12)^2}-\tfrac32$$
as the most likely number of elements in the union, assuming remainder terms are negligible.
More analysis of the remainder terms is required.

We have seen that when $s_1=s_3$, the remainder terms can apparently only affect the value of $N$ by 1. In Table \ref{D8} we present data when $s_3=s_1+8$,
indicating rather good agreement. Here ``actual $N$" is where the maximum actually occurs.

\begin{table}
\begin{center}
\caption{Comparison of actual $N$ and formula $N$}
\label{D8}
\begin{tabular}{rrr|cl}
$s_1$&$s_3$&$s_2$&\text{actual }$N$&\qquad(\ref{Nform})\\
\hline
                       $4$&$ 12$&$ 4$&$ 4$&$ 2.20621862$\\
                       $4$&$ 12$&$ 5$&$ 4$&$ 2.94878520$\\
                       $4$&$ 12$&$ 6$&$ 5$&$ 3.64427035$\\
                       $4$&$ 12$&$ 7$&$ 6$&$ 4.29480265$\\
                       $4$&$ 12$&$ 8$&$ 6$&$ 4.90267008$\\
                       $4$&$ 12$&$ 9$&$ 7$&$ 5.47025916$\\
                       $4$&$ 12$&$ 10$&$ 7$&$ 6.00000000$\\
                       $4$&$ 12$&$ 11$&$ 7$&$ 6.49431892$\\
                       $4$&$ 12$&$ 12$&$ 8$&$ 6.95559936$\\
                       $4$&$ 12$&$ 13$&$ 8$&$ 7.38615134$\\
                       $4$&$ 12$&$ 14$&$ 9$&$ 7.78818860$\\
                       $4$&$ 12$&$ 15$&$ 9$&$ 8.16381295$\\
                       $4$&$ 12$&$ 16$&$ 9$&$ 8.51500450$\\
                       $4$&$ 12$&$ 17$&$ 9$&$ 8.84361678$\\
                      $4$&$ 12$&$ 18$&$ 10$&$ 9.15137575$\\
                      $4$&$ 12$&$ 19$&$ 10$&$ 9.43988174$\\
                      $4$&$ 12$&$ 20$&$ 10$&$ 9.71061354$\\
                      \hline
                       $12$&$ 20$&$ 4$&$ 4$&$ 3.26186337$\\
                       $12$&$ 20$&$ 5$&$ 5$&$ 4.11613751$\\
                       $12$&$ 20$&$ 6$&$ 6$&$ 4.94207761$\\
                       $12$&$ 20$&$ 7$&$ 6$&$ 5.74010932$\\
                       $12$&$ 20$&$ 8$&$ 7$&$ 6.51071593$\\
                       $12$&$ 20$&$ 9$&$ 8$&$ 7.25443290$\\
                      $12$&$ 20$&$ 10$&$ 9$&$ 7.97184215$\\
                      $12$&$ 20$&$ 11$&$ 9$&$ 8.66356603$\\
                      $12$&$ 20$&$ 12$&$ 10$&$ 9.33026127$\\
                      $12$&$ 20$&$ 13$&$ 11$&$ 9.97261301$\\
                     $12$&$ 20$&$ 14$&$ 11$&$ 10.59132893$\\
                     $12$&$ 20$&$ 15$&$ 12$&$ 11.18713359$\\
                     $12$&$ 20$&$ 16$&$ 12$&$ 11.76076312$\\
                     $12$&$ 20$&$ 17$&$ 13$&$ 12.31296031$\\
                     $12$&$ 20$&$ 18$&$ 13$&$ 12.84446999$\\
                     $12$&$ 20$&$ 19$&$ 14$&$ 13.35603495$\\
                     $12$&$ 20$&$ 20$&$ 14$&$ 13.84839221$
\end{tabular}
\end{center}
\end{table}

\end{document}